\documentclass[11pt,reqno]{amsart}
\usepackage{amssymb, bbm, color}
\usepackage[makeroom]{cancel} 
\usepackage{amsfonts}
\usepackage{mathrsfs}
\usepackage{amsmath}
\usepackage{graphicx}
\usepackage{hyperref}
\usepackage{float}
\usepackage{epstopdf}
\usepackage{color}
\usepackage{bm}
\usepackage{comment}
\usepackage{soul}
\usepackage{esint}
\usepackage[shortlabels]{enumitem}
\usepackage{eufrak}
\allowdisplaybreaks
\newtheorem{theorem}{Theorem}[section]

\newtheorem{lemma}[theorem]{Lemma}
\newtheorem{corollary}[theorem]{Corollary}
\newtheorem{remark}[theorem]{Remark}

\newtheorem{definition}[theorem]{Definition}

\newtheorem{assumption}[theorem]{Assumption}

\def\mcE{\mathcal{E}}

\def\R{{\mathbbm R}}

 \makeatletter
\@namedef{subjclassname@2020}{%
  \textup{2020} Mathematics Subject Classification}
\makeatother

\def\<{\langle}
\def\>{\rangle}

\numberwithin{equation}{section}

\begin{document}
\title[On Kigami's conjecture of the embedding $\mathcal{W}^p(K)\subset C(K)$]{On Kigami's conjecture of the embedding $\mathcal{W}^p(K)\subset C(K)$}

\author{Shiping Cao}
\address{Department of Mathematics, University of Washington, Seattle, WA 98195, USA}
\email{spcao@uw.edu}
\thanks{}

\author{Zhen-Qing Chen}
\address{Department of Mathematics, University of Washington, Seattle, WA 98195, USA}\email{zqchen@uw.edu}
\thanks{}

\author{Takashi Kumagai}
\address{Department of Mathematics, Waseda University, Tokyo 169-8555, Japan}
\email{t-kumagai@waseda.jp}
\thanks{} 


\date{}


\maketitle

\begin{abstract}
Let $(K,d)$ be a connected compact metric space and $p\in (1, \infty)$. 
Under the assumption of \cite[Assumption 2.15]{Ki2} and the conductive 
$p$-homogeneity, 
we show that $\mathcal{W}^p(K)\subset C(K)$ holds if and only if $p>\operatorname{dim}_{AR}(K,d)$, where $\mathcal{W}^p(K)$ is Kigami's $(1,p)$-Sobolev space and  $\operatorname{dim}_{AR}(K,d)$ is the Ahlfors regular dimension.
\end{abstract}

\section{Introduction}\label{sec1}

The classical Sobolev spaces $W^{1, p}(\R^d)$ on $\R^d$ are vector spaces   consisting of  
$\mathbb L^p$-integrable functions whose distributional derivatives are  also $\mathbb L^p$-integrable. They play a fundamental role in analysis and partial differential equations (PDEs in abbreviation). 
 For instance, many important PDEs have weak solutions in suitable Sobolev spaces even though they may not have classical or strong solutions. 
A  classical Sobolev embedding theorem asserts that  $W^{1,p}(\mathbb R^d)\subset C(\mathbb R^d)$ if and only if $p>d$.
 
 \medskip
 
 Recently, there are many progresses in the  
analysis on metric measure spaces.  One major endeavor is to define and study 
 Sobolev spaces on non-smooth metric measure spaces.  
  A natural approach to define Sobolev spaces on metric measure spaces is to use the upper gradient, which corresponds to local Lipschitz constants for Lipschitz functions on $\R^d$;   See Haj\l asz
    \cite{H}, Cheeger \cite{Cheeger} and Shanmugalingam \cite{Shan} for pioneering works in this direction. 
   However, on 
 most  fractals,  including generalized Sierpi\'nski carpets and high dimensional Sierpi\'nski gaskets, recent studies \cite{KM,KM2} by Kajino and Murugan 
 shows that such an approach does not work; see \cite[Pages  2-3]{Ki2} for details. 
In fact, on fractals the Sobolev spaces defined in terms of  upper gradients are typically the same as the $\mathbb L^p$-spaces  due to \cite[Proposition 7.1.33]{HKST} (see also \cite[Remark 11.7]{MS} for more details in the case of Sierpinski carpets).  
Motivated by the graph approximation  approach of Brownian motion 
 on strongly recurrent Sierpi\'nski carpets 
by Kusuoka and Zhou \cite{KZ}, recently Kigami \cite{Ki2} defined a function space $\mathcal{W}^p(K)$, which is the counterpart of $W^{1,p}(\R^d)$,  
 on a class of compact metric measure spaces  $(K, d)$ in terms of  the scaling limit of $p$-energies defined on an increasing sequence of
   discrete graph approximations of the metric space $(K, d)$. In the same paper, it is conjectured \cite[Conjecture on page 9]{Ki2}   that  
   $\mathcal{W}^p(K)$ can be embedded into  the space of continuous functions on $K$ if and only if $p$ is  strictly larger than the 
    Ahlfors regular conformal dimension of $(K,d)$.
  The purpose of  this paper is to settle  this  conjecture affirmatively.

Let us very briefly overview the history of  analysis on infinitely ramified fractals. 
On the Sierpi\'nski carpets, Barlow and Bass \cite{BB,BB3} constructed 
Brownian motion and obtained detailed two-sided heat kernel estimates.
The uniqueness of Brownian motion on the carpets was proved in \cite{BBKT}. 
In \cite{KZ}, Kusuoka and Zhou introduced a way of constructing local regular self-similar Dirichlet forms on a class of self-similar sets, including the 
strongly recurrent Sierpi\'nski carpets, 
as the limit of discrete energies on graph partitions of the fractals. In recent years, on the Sierpi\'nski carpet itself \cite{shi}, and on other classes of Sierpi\'nski carpet like fractals, including unconstrained Sierpi\'nski carpets \cite{CQ}, 
 chipped    Sierpi\'nski carpets 
 and crosses \cite{Ki2}, and 
hollow Sierpi\'nski carpets \cite{CQ2}, new progresses have been made in defining self-similar Dirichlet forms or more generally $p$-energies. A key step involved these examples is that one can derive $\sigma_m\lesssim R_m$ (which are  essentially the conditions (B1) and (B2) of \cite{KZ}), where $R_m$ and $\sigma_m$, $m\geq 1$ are two classes of constants in  \cite{KZ}, using the fact that functions in the domain of the Dirichlet forms are H\"older continuous in low dimensional cases. 

 In \cite{Ki2}, Kigami introduced the condition `$p$-conductive homogeneity' (see Definition \ref{def256}) as a generalization of $\sigma_m\lesssim R_m$ for $p\in (1,\infty)$. 
Let  
\begin{eqnarray*}
	\operatorname{dim}_{AR}(K,d)&:=&\inf\{\alpha: \mbox{there exists a metric $\rho$ on $K$ that is quasi-symmetric to $d$ and}\\ 
	&& \hskip 0.1truein \mbox{a Borel regular measure $\mu$ that is $\alpha$-Ahlfors regular
		with respect to $\rho$.}\}
\end{eqnarray*}
$\operatorname{dim}_{AR}(K,d)$ is called the Ahlfors regular conformal dimension of $(K,d)$ 
(see \cite[page 4]{Ki2}). When $p>\operatorname{dim}_{AR}(K,d)$, it follows from \cite[Theorems 4.7.6 and 4.9.1]{Ki1} that every function in the space  $\mathcal{W}^p(K)$ of the $p$-energy has a  H\"older continuous modification. 
On the other hand,  not much is known about the case $p\leq \operatorname{dim}_{AR}(K,d)$, even   under
 the conductive homogeneity condition. 
 In this note, assuming $K$ is compact and that the $p$-conductive homogeneity holds, 
we show that  $\mathcal{W}^p(K)$ can not be embedded into $C(K)$  when  $p\leq  \operatorname{dim}_{AR}(K,d)$.
More precisely,   for $p\leq \operatorname{dim}_{AR}(K,d)$, we construct 
a function in $\mathcal{W}^p(K)$ that is unbounded on $K$.   This affirmatively answers a conjucture by Kigami  \cite[Conjecture on page 9]{Ki2}. 
We  point out that, as mentioned in \cite{Ki2}, a more fundamental question is whether  $\mathcal{W}^p(K)\subset C(K)$ is a dense subset of $\mathbb L^p(K,\mu)$,
which   remains open  largely due to the lack of  an elliptic Harnack principle of $p$-harmonic functions on the corresponding graphs.

 \begin{theorem}
Let $(K,d)$ be 
a connected compact 
metric space satisfying \cite[Assumption 2.15]{Ki2}, let $1<p<\infty$, 
and assume that the $p$-conductive homogeneity 
holds in the sense of Definition \ref{def256}. 
 Then, the embedding $\mathcal{W}^p(K)\subset C(K)$ holds if and only if $p>\operatorname{dim}_{AR}(K,d)$.
\end{theorem}

\section{Settings}
In this section, we  recall   the basic setting from Kigami's frameworks \cite{Ki1,Ki2}. We begin with a review of the definition of trees with a reference point \cite[Definition 2.1 and 2.2]{Ki2}.

\begin{definition} \rm
	Let $T$ be a countably infinite set, and let $E\subset T\times T\setminus \{(w,w):w\in T\}$ be a collection of edges on $T$ which satisfies that $(v,w)\in E$ if $(w,v)\in E$.  
	\begin{enumerate}
		\item The graph $(T,E)$ is called locally finite if $\#\{v\in T:(w,v)\in E\}<\infty$ for any $w\in T$. 
		
		\item For $w_0,\cdots,w_n\in T$,  $(w_0,w_1,\cdots,w_n)$ is called a path between $w_0$ and $w_n$ if and only if $(w_i,w_{i+1})\in E$ for any $i=0,1,\cdots,n-1$. A path $(w_0,w_1,\cdots,w_n)$ is called simple if and only if $w_i\neq w_j$ for any $0\leq i<j\leq n$.
		
		\item $(T,E)$ is called a tree if and only if there exists a unique simple path between $w$ and $v$ for any $w,v\in T$ with $w\neq v$. For a tree $(T,E)$, the unique simple path between two vertices $w$ and $v$ is called the geodesic   between $w$ and $v$ and denoted by $\overline{wv}$. 
	\end{enumerate}
\end{definition}

\begin{definition} \rm
Let $(T,E)$ be a tree and let $\emptyset\in T$. The triple $(T,E,\emptyset)$ is called a tree with a reference point $\emptyset$. 
\begin{enumerate}
	\item Define $\pi:T\to T$ by 
	\[
	\pi(w)=
	\begin{cases}
		w_{n-1}\quad &\hbox{ if }w\neq\emptyset\hbox{ and }\overline{\emptyset w}=(\emptyset,w_1,\cdots,w_{n-1},w),\\
		\emptyset\quad&\hbox{ if }w=\emptyset.
	\end{cases}
	\]
	and set $S(w)=\{v\in T:\pi(v)=w\}\setminus\{w\}$. Moreover, we write $S^1(w)=S(w)$, and inductively, for $k\geq 2$, 
	we set $S^k(w)=\bigcup_{v\in S(w)}S^{k-1}(v)$. 
	
	\smallskip
	
	\item For $w\in T$ and $m\geq 0$, we define $|w|=\min\{n\geq 0:\pi^n(w)=\emptyset\}$ and $T_n=\{w\in T:|w|=n\}$. 
\end{enumerate}
\end{definition}

Next, we recall the definition of partition of a metric space from \cite[Definition 2.3]{Ki2} in a slightly stronger manner in that we require $K$ to be compact connected metric space instead of compact metrizable topological space having 
no isolated points. This is reasonable 
as the running assumption   of \cite{Ki2} is $K$ is a compact connected metric space; see \cite[line -3 on page 3]{Ki2}.

\begin{definition}\label{def23} \rm
Let $T$ be a countably infinite set, $(T,E)$ be a locally finite tree with no  leaves,  
and $\emptyset$ be the reference point of $(T,E)$. Let $(K,d)$ be a compact connected metric space. A collection of non-empty compact subsets $\{K_w\}_{w\in T}$ of $K$ is said to be  a partition of $(K, d)$ parameterized by $(T,E,\emptyset)$
   if and only if  it satisfies the following conditions (P1) and (P2).
\begin{enumerate}[(P1)]
	\item $K_\emptyset=K$ and for each $w\in T$, $K_w$ is connected  and $K_w=\bigcup_{v\in S(w)}K_v$.  
	
	\item For any $\omega\in \Sigma:=\big\{(\omega(i))_{i\geq 0}\subset T:\omega(i)=\pi(\omega(i+1))\hbox{ for any }i\geq 0\big\}$, $\bigcap_{i=0}^\infty K_{\omega(i)}$ is a single point.
\end{enumerate}
\end{definition}

We also recall the following definition of covering system from \cite[Definitions 2.26 and 2.29]{Ki2} for the definition of neighbor disparity constants.
	
\begin{definition}\label{def:2.4} \rm 
Assume that $\{K_w\}_{w\in T}$ is a partition of a metric space $(K,d)$ parameterized by $(T,E,\emptyset)$.  
\begin{enumerate}[\rm (i)]
\item Let $\{G_i\}_{i=1}^k$ be a collection of subsets of $T_n$. The family $\{G_i\}_{i=1}^k$ is called a covering of $A\subset T_n$ with covering numbers $(N_T,N_E)$ if 
\[
A=\bigcup_{i=1}^kG_i,\quad \max_{w\in A}\{1\leq i\leq k:\,w\in G_i\}\leq N_T, 
\] 
and for any $u,v\in A$ such that $K_u\cap K_v\neq\emptyset$, there exists an integer $l\leq N_E$ and $w_1,w_2,\cdots,w_{l+1}\in A$ such that $K_{w_j}\cap K_{w_{j+1}}\neq\emptyset$ and $\{w_j,w_{j+1}\}\subset G_i$ for some $i\in \{1,2,\cdots,k\}$ for every $j\in \{1,2,\cdots,l\}$. 

\item We call $\mathscr{I}\subset \bigcup_{n=0}^\infty\{A:\,A\subset T_n\}$ a covering system with covering number $(N_T,N_E)$ if the following are satisfied
\begin{enumerate}
	\item[(ii.1)] $\sup_{A\in\mathscr{I}}\#A<\infty$, where $\#A$ is the cardinality of $A$.
	
	\item[(ii.2)] For any $w\in T$ and $m\geq 1$, there exists a finite subset $\mathscr{N}\subset \mathscr{I}$ such that $\mathscr{N}$ is a covering of $S^m(w)$ with covering numbers $(N_T,N_E)$.
	
	\item[(ii.3)] For any $G\in \mathscr{I}$, there exists a finite subset $\mathscr{N}\subset \mathscr{I}$ such that $\mathscr{N}$ is a covering of $S^m(G):=\cup_{w\in G} S^m (w)$ with covering numbers $(N_T,N_E)$.
\end{enumerate}
\end{enumerate}
\end{definition}

 Write 
\[
\begin{aligned}
	\Gamma_{M}(w)=\{v\in T_{|w|}:\,&\hbox{there exist }k\leq M\hbox{ and }w_0,w_1,\cdots,w_k\in T_{|w|}\hbox{ such that }\\&w=w_{0},\,v=w_{k}\text{ and }K_{w_{i}}\cap K_{w_{i+1}}\neq\emptyset\hbox{ for }i=0,1,\cdots,k-1\}
\end{aligned}
\]
for $M\geq 1$ and $w\in T$. We consider the following conditions, 
which is a subset of \cite[Assumptions 2.6,\,2.7,\,2.10 and 2.12]{Ki2}.

\begin{assumption}\label{assump24}
	Let $T$ be a countably infinite set and $(T,E)$ be a locally finite tree with no leaves and $\emptyset$ be the reference point of $(T,E)$. Let $(K,d)$ be a compact connected metric space with $diam(K, d):=\sup_{x,y\in K}d(x,y)=1$. In addition, let $\{K_w\}_{w\in T}$ be a partition of $(K,d)$ parameterized by $(T,E,\emptyset)$. 
	\begin{enumerate}	[\rm (i)]
		\item There exists $L_*\in (0,\infty)$ such that 
		\[\#\{ v\in T_n\setminus\{w\}:K_w\cap K_v\neq \emptyset\}\leq L_*\quad \hbox{ for each }n\geq 0\hbox{ and }w\in T_n.\]  
		
		\item There exists $M_*\geq 1$ such that 
		\[
		\pi\big(\Gamma_{M_*+1}(w)\big)\subset \Gamma_{M_*}\big(\pi(w)\big).
		\]
		
		\item $\mu$ is a Borel regular probability measure on $(K,d)$ that satisfies
		\[
	    0<\mu(K_w)=\sum_{v\in S(w)}\mu(K_v)\quad\hbox{ for each }w\in T,
		\]
		and there are $c_\mu\in (0,\infty)$ and $\gamma\in (0,1)$ such that 
		\[
		\mu(K_w)\leq c_\mu\gamma^{|w|}\quad\hbox{ for each }w\in T.
		\] 
	\end{enumerate}
\end{assumption}\medskip

With the measure $\mu$ in Assumption \ref{assump24}(iii), for $f\in \mathbb L^1(K,\mu)$ and $n\geq 0$, we define $P_nf\in l(T_n)$ by 
\[
P_nf(w)=\frac{1}{\mu(K_w)}\int_{K_w}f(x)\mu(dx)\quad\hbox{ for each }w\in T_n. 
\]\medskip 

We next define the graph energy on $T_n$. Denote by  $l(A):=\{f:A\to \R\}$  the space of functions on $A$. 

\begin{definition}\label{def25} \rm
Assume that $\{K_w\}_{w\in T}$ is a partition of a metric space $(K,d)$ parameterized by $(T,E,\emptyset)$. Let $\mathscr{I}\subset \bigcup_{n=0}^\infty\{A:\,A\subset T_n\}$ be a covering system with covering number $(N_T,N_E)$. 
  For $p\in [1, \infty)$, define 
 the discrete $p$-energy on $A\subset T_n,n\geq 1$ by
\[
\mathcal{E}_{p,A}^{n}(f)=\frac12\sum_{w,v\in A:K_w\cap K_v\neq \emptyset}
\big| f(w)-f(v)\big|^p  
\quad\hbox{ for }f\in l(A).
\]
For short, when $A=T_n$,  write $\mathcal{E}_{p}^{n}(f)$ for $\mathcal{E}_{p,T_n}^{n}(f)$. 
Define effective conductance and neighbor disparity constant as follows. 

\begin{enumerate} [\rm (i)]
	\item For $n\geq 1,m\geq 0$ and disjoint $A_1,A_2\subset T_n$, define the effective conductance   
	\[
	\mathcal{E}_{p,m}(A_1,A_2,T_n) :=\inf\{\mathcal{E}_{p}^{n+m}(f):\,f\in l(T_{m+n}),\,f|_{S^m(A_1)}=1,\,f|_{S^m(A_2)}=0\}.
	\]
	
	\item For $n\geq 1,m\geq 0$ and $A\subset T_n$, define  the neighbor disparity constant to be
	\[ 
	\sigma_{p,m}(A) :=\sup_{f\in\mathbb L^1(K;\mu)}\frac{\mcE_{p,A}^n\big((P_nf)|_A\big)}{\mcE_{p,S^m(A)}^{n+m}\big((P_{n+m}f)|_{S^m(A)}\big)}.
	\]
\end{enumerate}
\noindent Define
\begin{align*}
\sigma_{p,m,n}&=\max\limits_{A\in\mathscr{I_n},\,A\subset T_n}\sigma_{p,m}(A)
\quad \mbox{ and } \quad  
\sigma_{p,m}=\sup_{n\ge 1}\sigma_{p,m,n},\\
\mathcal{E}_{M,p,m}&=\sup_{w\in T, |w|\ge 1}\mathcal{E}_{p,m}(w,T_{|w|}\setminus\Gamma_{M}(w),T_{|w|}).\end{align*}
\end{definition} 

 Hereafter, we fix a covering system $\mathscr{I}$ with covering number $(N_T,N_E)$. Under mild assumptions on $\sigma_{p,m,n}$ and $\mathcal{E}_{p,m}(w,T_{|w|}\setminus\Gamma_{M}(w),T_{|w|})$, we will prove in Section \ref{sec3} that there is an unbounded function with uniformly bounded  discrete $p$-energies. See the precise statement in the theorem below. 

\begin{theorem}  \label{T:2.6}
Suppose that Assumption \ref{assump24} holds and that 
\begin{eqnarray}
	&\mathcal{E}_{p,m}(w,T_{|w|}\setminus\Gamma_{M_*}(w),T_{|w|})\leq c\sigma^{-m}\ &\hbox{ for all } m\ge 0, w\in T,\label{eq:cond1}\\
	&\sigma_{p,m,n}\leq c\sigma^{m}\ &\hbox{ for all }m\ge 0, n\geq 1.\label{eq:cond2}
\end{eqnarray}
hold with $\sigma\leq 1$, where $M_*\geq 1$ is the constant in Assumption  \ref{assump24}(ii). 
For $p \in (1, \infty)$, there is an unbounded $f\in \mathbb L^p(K;\mu)\setminus C(K)$ such that $\sup\limits_{n\geq 1}\sigma^n\mcE^n_p(P_nf)<\infty$. 
\end{theorem}

\medskip

A sufficient condition for \eqref{eq:cond1}-\eqref{eq:cond2} is given in Remark \ref{Junthm} below. To state it, we need a definition.

\begin{definition}{\rm (Conductive Homogeneity 
\cite[Definition 3.4]{Ki2})} \rm 
\label{def256}
Let $p\in [1, \infty)$. A compact metric space $K$ with a partition $\{K_w\}_{w\in T}$ and a measure $\mu$ is said to be $p$-conductively homogeneous if and only if 
\[\sup_{m\ge 0}\sigma_{p,m}\mathcal{E}_{M_*,p,m}<\infty,\] 
where $M_*$ is the same constant as given in Assumption \ref{assump24}\,(ii).
\end{definition}

\begin{remark}\label{Junthm}   Suppose that \cite[Assumptions 2.6, 2.7, 2.10 and 2.12]{Ki2} hold and $p\in (1, \infty)$.
It is shown in    \cite[Theorem 1.3]{Ki2} 
that  the connected compact metric space $(K, d)$ 
is $p$-conductively homogeneous if and only if there exist $c_1,c_2>0$ and $\sigma>0$ such that the followings hold for any $m\ge 0, n\ge 1$ and $w\in T$,
\begin{equation} \label{e:2.1} 
c_1\sigma^{-m}\le \mathcal{E}_{p,m}(w,T_{|w|}\setminus\Gamma_{M_*}(w),T_{|w|})\le c_2\sigma^{-m}
\quad \hbox{and} \quad 
c_1\sigma^m\le \sigma_{p,m,n}\le c_2\sigma^m.
\end{equation}
\end{remark}

\medskip

\begin{definition} \rm 
Assume that \eqref{e:2.1} holds. For $p\in (1, \infty)$, define 
\begin{equation}\label{eq:W^pdef}
\mathcal{W}^p(K):= \Big\{f:f\in \mathbb L^p(K;\mu) \hbox{ with } \sup_{n\geq 1}\sigma^n\mcE^n_p(P_nf)<\infty \Big\},
\end{equation}
where $\sigma$ is the constant of \eqref{e:2.1}. 
\end{definition}

The following corollary to Theorem \ref{T:2.6} answers affirmatively the conjecture on page 9 of \cite{Ki2}.

\begin{corollary}\label{coro}
Suppose that $p\in (1, \infty)$, and the connected compact metric space $(K,d)$ is 
conductively $p$-homogeneous and satisfies \cite[Assumption 2.15]{Ki2}. Then
 $\mathcal{W}^p(K)\subset C(K)$ if and only if $p>\operatorname{dim}_{AR}(K,d)$. 
\end{corollary}

\begin{proof}
First of all, we remark that \cite[Assumptions 2.6,\,2.10 and 2.12]{Ki2} are consequences of
(1)-(4) of 
 \cite[Assumption 2.15]{Ki2} by \cite[Proposition 2.16]{Ki2}, and \cite[Assumption 2.7]{Ki2} is
  just \cite[Assumption 2.15(5)]{Ki2}.
   So 
     \cite[Theorem 1.3]{Ki2} 
   holds under  \cite[Assumption 2.15]{Ki2} and, consequently, 
    $\mathcal{W}^p(K)$ is well defined. Moreover, in view of Remark \ref{Junthm}, 
    \eqref{eq:W^pdef} holds under  \cite[Assumption 2.15]{Ki2}. 
	
By  \cite[Proposition 3.3]{Ki2},  
\begin{equation}\label{e:2.4}
p>\operatorname{dim}_{AR}(K,d)
\quad \hbox{if and only if} \quad  \sigma>1,
\end{equation} 
  where $\sigma$ is the constant of \eqref{e:2.1}.   
 By \eqref{e:2.4} and Theorem \ref{T:2.6},  $\mathcal{W}^p(K)\not\subset C(K)$ if $p\leq\operatorname{dim}_{AR}(K,d)$.
On the other hand,  it follows directly from 
\cite[Lemma 3.34]{Ki2}
and \eqref{e:2.4} that $\mathcal{W}^p(K)\subset C(K)$ if $p>\operatorname{dim}_{AR}(K,d)$. 
This combined with Theorem \ref{T:2.6} shows that 
$\mathcal{W}^p(K)\subset C(K)$ if and only if $p>\operatorname{dim}_{AR}(K,d)$. 
 \end{proof}

\begin{remark} \rm 
 Quite recently, Murugan and Shimizu \cite{MS} constructed a  
$p$-energy on planar Sierpi\'nski carpets for all $p\in (1, \infty)$. 
While they do not use the $p$-conductive homogeneity in the construction, 
they prove that the planar Sierpi\'nski carpet equipped with the self-similar measure with the equal weight is $p$-conductive homogeneous for 
any $p\in (1, \infty)$ --  
see \cite[Theorem C.28]{MS}. This gives  a non-trivial example that enjoys
$p$-conductive homogeneity for every $p\le \operatorname{dim}_{AR}(K,d)$. 
\end{remark}

\section{Proof of Theorem \ref{T:2.6}}\label{sec3}
In this section, we prove Theorem \ref{T:2.6}. For the convenience of the reader, we  illustrate 
the basic ideas behind our proof 
in the context of the classical  Sobolev space $W^{1,p}(\R^d)$ with $d\geq 2$ and $1\leq p\leq d$.
 Consider $f_n\in C_c^\infty(\R^d)$ such that 
\[
f_n|_{B(0,2^{-n})}\equiv 1,\quad f_n|_{\R^d\setminus B(0,2^{-n+1})} 
\quad \hbox{ and } \quad    | \nabla f_n (x)| \leq c_1 2^{n}.  
\]
Note  $\int_{\R^d}|\nabla f(x)|^pdx\leq c_2  2^{-n(d-p)}$ for $n\geq 1$. 
Define $f=\sum_{n=1}^\infty \frac{1}{n}f_n$. 
Then 
\[
\int_{\R^d}|\nabla f(x)|^pdx=\sum_{n=1}^\infty n^{-p}\int_{\R^d}|\nabla f_n(x)|^pdx \leq c_2 \sum_{n=1}^\infty n^{-p} 2^{-n(d-p)}<\infty
\]
 for any   $p\in [1, d]$  with $d\geq 2$. 
 On the other hand,    $f$ has a logarithmic singularity at $0$ and so $f$ is unbounded near $0$. 
This shows that $W^{1,p}\setminus C(\R^d)\supset \{f\}$ is non-empty. 
In the setting of partitions of metric spaces, we will modify the above approach by constructing the function $f$  as the weak limit of a sequence of functions.

\medskip

Now we proceed to give the proof of Theorem \ref{T:2.6}. First, we note the following.

\begin{lemma}\label{lemma31} 
Under Assumption \ref{assump24}(ii), we have 
\begin{equation}\label{e:3.1}
\pi^{k}\big(\Gamma_{M_*+k}(w)\big)\subset \Gamma_{M_*}\big(\pi^k(w)\big) \quad\hbox{ for any }w\in T\text{ and }k\geq 1.
\end{equation}
\end{lemma}

\begin{proof}
First, we show that  
\begin{equation}\label{e:3.2}
\pi\big(\Gamma_{i}(w)\big)\subset \Gamma_{i}\big(\pi(w)\big)\quad\hbox{ for any }w\in T\text{ and }i\geq 1.
\end{equation}
In fact, $v\in \pi\big(\Gamma_{i}(w)\big)$ implies that there exist $w_0,w_1,\cdots,w_i\in T_{|w|}$ such that $w_0=w,v=\pi(w_i)$ and $K_{w_j}\cap K_{w_{j+1}}\neq\emptyset$ for  
$j=0,1,\cdots,i-1$. So by 
letting $v_j=\pi(w_j)$ for $j=0,1,\cdots,i$, we see that $v_0=\pi(w)$, $v_i=v$ and $K_{v_j}\cap K_{v_{j+1}}\supset K_{w_j}\cap K_{w_{j+1}}\neq \emptyset$ for $i=0,1,2,\cdots,i-1$, which means that $v\in \Gamma_{i}\big(\pi(w)\big)$. 
	
We prove \eqref{e:3.1} by induction. First note that \eqref{e:3.1}  holds for $k=1$ by Assumption \ref{assump24}(ii). Now suppose  that  \eqref{e:3.1} holds for $k-1$ for some $k\geq 2$.
Then 
\[
\begin{aligned}
\pi^{k}\big(\Gamma_{M_*+k}(w)\big)&=\pi^{k}\big(\bigcup_{v\in \Gamma_{k-1}(w)}\Gamma_{M_*+1}(v)\big)\\
&=\bigcup_{v\in \Gamma_{k-1}(w)}\pi^{k}\big(\Gamma_{M_*+1}(v)\big)\\
&\subset \bigcup_{v\in \Gamma_{k-1}(w)}\pi^{k-1}\big(\Gamma_{M_*}\big(\pi(v)\big)\big)\\
&=\pi^{k-1}\Big(\Gamma_{M_*}\big(\pi(\Gamma_{k-1}(w))\big)\Big)\\
&\subset \pi^{k-1}\Big(\Gamma_{M_*+k-1}\big(\pi(w)\big)\Big)\\
&\subset \Gamma_{M_*}\big(\pi^{k-1}\circ\pi(w)\big)=\Gamma_{M_*}\big(\pi^k(w)\big),
\end{aligned}
\]
where in the third line we used Assumption \ref{assump24}(ii), in the fifth line we used 
\eqref{e:3.2} with $i=k-1$ and in the last line we used the induction hypothesis that the desired formula holds for $k-1$.
\end{proof}\medskip

\begin{proof}[Proof of Theorem \ref{T:2.6}]
Fix $x\in K$, and we pick $\omega\in \Sigma$ such that $x=\bigcap_{i=0}^\infty K_{\omega(i)}$. Let $w_j=\omega(jM_*+j)$ for $j\geq 0$, so 
\begin{equation}\label{eqn31}
\pi^{M_*+1}(w_{j+1})=w_j\quad\hbox{ for }j\geq 0. 
\end{equation}
Let 
\[
A_j=\Gamma_{M_*}(w_j),\ B_j=\Gamma_{2M_*}(w_j)\ \hbox{ and }B^*_j=\Gamma_{2M_*+1}(w_j).
\]
Then  for $j\geq 1$, 
\begin{equation}\label{eqn32} 
 \pi^{M_*+1}(B^*_{j+1}) =\pi^{M_*+1}\big(\Gamma_{2M_*+1}(w_{j+1})\big) 
 \subset \Gamma_{M_*}\big(\pi^{M_*+1}(w_{j+1})\big) 
 =\Gamma_{M_*}(w_j)   =A_{j}, 
 \end{equation}
where we used the definition of $B^*_{j+1}$ 
in the first equality,  \eqref{e:3.1} in the inclusion,    \eqref{eqn31} in the second equality, and  
the definition of $A_j$ in the last equality.  

For $k\geq 1$ and $k(M_*+1)\leq n<(k+1)(M_*+1)$,   define 
\begin{equation}\label{eqn33}
f_n=\sum_{j=1}^k\frac{1}{j}\,f_{n,j},
\end{equation}
where  $f_{n,j}\in l(T_n)$ for $1\leq j\leq k$  are such that $0\leq f_{n,j}\leq 1$, 
\begin{equation}\label{eqn34} 
f_{n,j}|_{T_n\setminus S^{n-(M_*+1)j}(B_j)}=0,\ f_{n,j}
|_{S^{n-(M_*+1)j}(A_j)}=1\,\text{ and }\mcE^n_p(f_{n,j})\leq C_1\sigma^{-n+(M_*+1)j}
\end{equation}
for some $C_1$ independent of $n,j$. 
Such $f_{n,j}$ exists. Indeed, for 
each $w\in A_j$, we can find $f_{n,j,w}$ such that $0\leq f_{n,j,w}\leq 1$ 
\[
f_{n,j,w}|_{S^{n-(M_*+1)j}(w)}=1,\  f_{n,j,w}|_{T_n\setminus S^{n-(M_*+1)j}\big(\Gamma_{M_*}(w)\big)}=0\hbox{ and }\mcE^n_p(f_{n,j,w})\leq c\sigma^{-n+(M_*+1)j},
\]
where $c$ is the constant of \eqref{eq:cond1}. 
Then    $f_{n,j} :=\max_{w\in A_j}f_{n,j,w}$ has the desired properties.  Note that 
the energy estimate of $f_{n,j}$   follows from the energy estimate of $f_{n,j,w},w\in A_j$ and Assumption \ref{assump24}(i):
\[
\mcE^n_p(f_{n,j})\leq \sum_{w\in A_j}\mcE^n_p(f_{n,j,w})\leq \# A_j\,c\sigma^{-n+(M_*+1)j}\leq (L_*+1)^{M_*}c\,\sigma^{-n+(M_*+1)j},
\]
where $\#A_j$ is the cardinality of $A_j$ and $L_*$ is the constant of Assumption \ref{assump24}(i).
 So the last property in  \eqref{eqn34} holds with $C_1=(L_*+1)^{M_*}c$.

Next, we derive an upper bound  for  $\mcE^{p}_n(f_n)$. First,  observe that
\[
\begin{aligned}
\Gamma_1\big(S^{n-(M_*+1)k}(B_k)\big)
&\subset S^{n-(M_*+1)k}\big(\Gamma_1(B_k)\big)
\\&\subset S^{n-(M_*+1)(k-1)}\Big(\pi^{M_*+1}\big(\Gamma_1(B_k)\big)\Big)
\\&=S^{n-(M_*+1)(k-1)}\big(\pi^{M_*+1}(B^*_k)\big)\subset S^{n-(M_*+1)(k-1)}(A_{k-1}),
\end{aligned}
\]
where we used the fact that $B_k^*=\Gamma_1(B_k)$ in the equality, and used \eqref{eqn32} in the last inclusion. 
Hence, as $f_n=\sum_{j=1}^kf_{n,j}$, $\sum_{j=1}^{k-1}j^{-1}f_{n,j}$ is a constant on $S^{n-(M_*+1)(k-1)}A_{k-1}$ and $f_{n,k}$ supports on $S^{n-(M_*+1)k}B_k$, we have
\[
\begin{aligned}
\mcE^{p}_n(f_n)&=\frac12\sum_{w,w'\in T_n, K_w\cap K_{w'}\neq\emptyset\atop \{w,w'\}\cap S^{n-(M_*+1)k}(B_k)\neq\emptyset}\big(f_n(w)-f_n(w')\big)^p+\frac12\sum_{w,w'\in T_n, K_w\cap K_{w'}\neq\emptyset\atop \{w,w'\}\cap S^{n-(M_*+1)k}(B_k)=\emptyset}\big(f_n(w)-f_n(w')\big)^p\\
&=\mcE^p_n(k^{-1} f_{n, k})+\mcE^p_n(\sum_{j=1}^{k-1}j^{-1}f_{n,j}).
\end{aligned}
\]
We can apply the same argument to $\sum_{j=1}^{k-1}j^{-1}f_{n,j}$ 
to conclude  that 
\[
\mcE^{p}_n(f_n)=\sum_{j=1}^k \mcE^{p}_n(j^{-1}f_{n,j})=\sum_{j=1}^k j^{-p}\mcE^{p}_n(f_{n,j})\leq \sum_{j=1}^k j^{-p}C_1\sigma^{-n+(M_*+1)j},
\]
where the inequality is due to \eqref{eqn34}. Hence, 
\begin{equation}\label{eqn35}
\sigma^n\mcE^{p}_n(f_n)\leq  C_1\sum_{j=1}^k j^{-p}C_1\sigma^{(M_*+1)j}\leq C_1\sum_{j=1}^\infty j^{-p}<\infty
\end{equation}
where we use the assumption that $\sigma\leq 1$. 

Next, we define $\hat{f}_n=\sum_{w\in T_n}f_n(w) {\mathbbm 1}_{K_w}$, where ${\mathbbm 1}_{K_w}$ is the indicator  function of  
$K_w$, i.e. ${\mathbbm 1}_{K_w}(x)=1$ if $x\in K_w$ and ${\mathbbm 1}_{K_w}(x)=0$ if $x\notin K_w$. By Assumption \ref{assump24}(iii), we see that $\hat{f}_n\in \mathbb L^p(K,\mu)$ and $P_n\hat{f}_n=f_n$. By \eqref{eqn35}, \cite[Lemma 2.27]{Ki2} (which is a $p$-version of \cite[Lemma 2.12]{KZ} based on the same proof) and \eqref{eq:cond2}, 
\begin{equation}\label{e:3.8}
\sigma^m\mcE^{p}_m(P_m\hat{f}_n)\leq 
C_2\sigma^n\mcE^{p}_n(f_n)\leq 
C_2 C_1\sum_{j=1}^\infty j^{-p}<\infty\quad\hbox{ for each }1\leq m\leq n,
\end{equation}
where $C_2=c(L_*)^{N_E}N_E^{p-1}N_T$ with $L_*$ being the constant in Assumption \ref{assump24}(ii) and $(N_T,N_E)$ being the covering numbers. In addition, by  Assumption \ref{assump24} and the definition of $f_n$ and $\hat{f}_n$, we have 
\begin{align*}
&\hat{f}_n|_{K_{w_k}}\geq \sum_{j=1}^{k}\frac{1}{j}\quad\hbox{ for every }k\geq 1,\,n\geq (M_*+1)k,\\
&\|\hat{f}_n\|_{\mathbb{L}^p(K;\mu)} =\|\sum_{w\in T_n}f_n(w) {\mathbbm 1}_{K_w}\|_{\mathbb{L}^p(K;\mu)}=\|\sum_{i=1}^k\sum_{w\in T_n}f_{n,j}(w) {\mathbbm 1}_{K_w}\|_{\mathbb{L}^p(K;\mu)}\\
&\qquad\qquad\quad\leq \sum_{j=1}^k\|\sum_{w\in T_n}f_{n,j}(w) {\mathbbm 1}_{K_w}\|_{\mathbb{L}^p(K;\mu)}\\
&\qquad\qquad\quad\leq \sum_{j=1}^k\mu(\bigcup_{w\in B_j}K_w)^{1/p}\leq \sum_{j=1}^\infty \big((L_*+1)^{2M_*}\,c_\mu \gamma^{jM_*+j}\big)^{1/p}<\infty\quad\hbox{ for every }n\geq 1,
\end{align*}
where $k$ is the same as in \eqref{eqn33} and we use the fact that $ \sum_{w\in T_n}f_{n,j}(w) {\mathbbm 1}_{K_w}$ takes values in $[0,1]$ and supports on $\bigcup_{w\in B_j}K_w$ by \eqref{eqn34}. 

Finally, we pick a $\mathbb L^p$-weakly convergent subsequence $\hat{f}_{n_i},i\geq 1$, and denote 
its weak limit in $\mathbb L^p(K,\mu)$ by $f$. Then, by \eqref{e:3.8}, we have 
\[
\sigma^m\mcE^{p}_m(P_mf)=\lim\limits_{i\to\infty}\sigma^m\mcE^{p}_m(P_m\hat{f}_{n_i})\leq C_1 C_2\sum_{j=1}^\infty j^{-p}<\infty\quad\hbox{ for each }m\geq 1,
\] 
where the equality holds because $\lim\limits_{i\to\infty}P_m\hat{f}_{n_i}=P_mf$ and $T_m$ is a finite set, and 
\[
f\geq \sum_{j=1}^{k}\frac{1}{j}\quad\mu\hbox{-a.e. on }K_{w_k}\hbox{ for every }k\geq 1,
\]
which is due to the fact that $\int_Ef(x)\mu(dx)=\lim\limits_{i\to\infty}\int_E\hat{f}_{n_i}(x)\mu(dx)\geq (\sum_{j=1}^{k}\frac{1}{j})\,\mu(E)$ for every  Borel $E\subset K_{w_k}$.
Since $\mu (K_{w_k})>0$ for each $k\geq 1$ by Assumption \ref{assump24} (iii), this completes the proof of the proposition as $f$ is unbounded 
on the compact set $K$
and hence $f$ can not have a continuous modification. 
\end{proof}

\vskip 0.2truein

\noindent {\bf Acknowledgements.}\,\,  
The authors are grateful to the referee and Ryosuke Shimizu for helpful comments and suggestions that help to improve  the exposition of this paper.
The research of Shiping Cao is partially supported by a grant from the Simons Foundation Targeted Grant (917524) to the Pacific Institute for the Mathematical Sciences. 
The research of Zhen-Qing Chen is partially supported by  a Simons Foundation fund.
The research of Takashi Kumagai is supported by JSPS
KAKENHI Grant Number 22H00099 and 23KK0050.

 \hskip 0.2truein
 

\begin{thebibliography}{10}
\bibitem{BB}
M.T. Barlow and R.F. Bass, \emph{The construction of Brownian motion on the Sierpinski carpet}, Ann. Inst.
Henri Poincar\'{e} 25 (1989), no. 3, 225--257.

\bibitem{BB3}
M.T. Barlow and R.F. Bass, \emph{Brownian motion and harmonic analysis on Sierpinski carpets,} Canad. J. Math. 51 (1999), no. 4, 673--744.

\bibitem{BBKT} M.T. Barlow, R.F. Bass, T. Kumagai and A. Teplyaev, \emph{Uniqueness of Brownian motion on Sierpinski carpets}, J. Eur. Math. Soc. 12 (2010), no. 3, 655--701.
	
\bibitem{CQ} 
S.~Cao and H.~Qiu, \emph{Dirichlet forms on unconstrained Sierpinski carpets}. ArXiv:2104.01529.

\bibitem{CQ2}
S.~Cao, H.~Qiu and Y.~Wang, \emph{Self-similar Dirichlet forms on polygon carpets}. ArXiv:2206.00040.

\bibitem{Cheeger}
J. ~Cheeger, \emph{Differentiability of Lipschitz functions on metric measure spaces}, Geom. Funct. Anal. 9 (1999), no. 3, 428--517.

\bibitem{H}
P. ~Haj\l asz,
\emph{Sobolev spaces on an arbitrary metric space}. Potential Anal. 5 (1996), no. 4, 403--415.

\bibitem{HKST}
J. ~Heinonen, P. ~Koskela, N. ~Shanmugalingam and J.T. ~Tyson,  
\emph{Sobolev spaces on metric measure spaces. An approach based on upper gradients}. 
New Math. Monogr., {\bf 27}, 
Cambridge University Press, Cambridge, 2015, xii+434 pp.


\bibitem{KM}
N. ~Kajino and M. ~Murugan, \emph{On the conformal walk dimension: quasisymmetric uniformization for symmetric diffusions,} Invent. Math. 231 (2023), no. 1, 263--405.

\bibitem{KM2} 
N. ~Kajino and M. ~Murugan, \emph{On the conformal walk dimension II: Non-attainment for some Sierpi\'nski carpets,} 2023, in preparation.

 
\bibitem{Ki1} 
J.~Kigami, \emph{Geometry and analysis of metric spaces via weighted partitions}. Lecture Notes in Mathematics, 2265. Springer, Cham, 2020, viii+162 pp. 

\bibitem{Ki2}
J.~Kigami, \emph{Conductive homogeneity of compact metric spaces and construction of p-energy}, Mem. Eur. Math. Soc., 5 (2023), viii+129 pp.

\bibitem{KZ}
S.~Kusuoka and X.~Zhou, \emph{Dirichlet forms on fractals: {P}oincar\'{e}
	constant and resistance}, Probab. Theory Related Fields \textbf{93} (1992),
no.~2, 169--196.

\bibitem{MS}
M. Murugan and R. Shimizu, \emph{First-order Sobolev spaces, self-similar energies and energy measures on the Sierpi\'{n}ski carpet}, arXiv:2308.06232.

\bibitem{Shan}
N. ~Shanmugalingam, \emph{Newtonian spaces: an extension of Sobolev spaces to metric measure spaces}, Rev. Mat. Iberoam. 16 (2000), no. 2, 243--279.

\bibitem{shi}
R. ~Shimizu, \emph{Construction of p-energy and associated energy measures on the Sierpi\'{n}ski carpet.} To appear in: Trans. Amer. Math. Soc. DOI: https://doi.org/10.1090/tran/9036

\end{thebibliography}
\end{document}